\documentclass[final]{lms}
\makeatletter
\gdef\@journal{%
  \vbox to 5.5\p@{\noindent
    \parbox[t]{\textwidth}{\raggedleft\normalfont\baselineskip 9\p@{\ }
  \vss}%
}}
\makeatother

\usepackage{amsmath,amssymb}
\usepackage[dvipdfm]{graphicx}
\newtheorem{thm}{Theorem}[section] 

\newtheorem{prop}[thm]{Proposition}

\newnumbered{defn}[thm]{Definition}
\newnumbered{rem}[thm]{Remark}

\newcommand{\Thm}[1]{Theorem~\ref{th:#1}}
\newcommand{\Prop}[1]{Proposition~\ref{prop:#1}}

\newcommand{\Eq}[1]{\eqref{eq:#1}}

\renewcommand{\a}{\alpha}\renewcommand{\b}{\beta}
\newcommand{\gm}{\gamma}
\newcommand{\eps}{\varepsilon}
\newcommand{\lm}{\lambda}

\newcommand{\om}{\omega}
\newcommand{\C}{\mathbb{C}}
\newcommand{\N}{\mathbb{N}}
\newcommand{\R}{\mathbb{R}}
\newcommand{\Z}{\mathbb{Z}}
\newcommand{\Log}{\operatorname{Log}}
\renewcommand{\Re}{\operatorname{Re}}\renewcommand{\Im}{\operatorname{Im}}

\title[Neo-classical inequality]
 {Fractional order Taylor's series and the neo-classical inequality} 

\author{Keisuke Hara and Masanori Hino}

\classno{26D15 (primary), 30B10, 26A33 (secondary)}

\extraline{The second author is supported by a Grant-in-Aid for Young Scientists (B) (21740094).}

\begin{document}
\maketitle

\begin{abstract}
We prove the neo-classical inequality with the optimal constant, which was conjectured by T.~J.~Lyons~[Rev. Mat. Iberoamericana~14 (1998) 215--310]. 
For the proof, we introduce the fractional order Taylor's series with residual terms.
Their application to a particular function provides an identity that deduces the optimal neo-classical inequality.
\end{abstract}
\section{Introduction}\label{sect:1}
In his celebrated study on the theory of rough paths~\cite{Ly98}, 
T.~J.~Lyons introduced the following neo-classical inequality, which was a key estimate to prove one of the fundamental theorems~\cite[Theorem~2.2.1]{Ly98}:
\begin{thm}[(Neo-classical inequality~{\cite[Lemma~2.2.2]{Ly98}})]\label{th:neo1}
Let $\a\in(0,1]$, $n\in\N=\{1,2,3,\dots\}$, $x\ge0$, and $y\ge0$.
Then, we have
\begin{equation}\label{eq:lyons}
\a^2\sum_{j=0}^n\binom{\a n}{\a j}x^{\a j}y^{\a(n-j)}\le (x+y)^{\a n}.
\end{equation}
\end{thm}
Here, in general, we define 
\begin{equation}\label{eq:binom}
\binom{w}{z}=\frac{\Gamma(w+1)}{\Gamma(z+1)\Gamma(w-z+1)}
\end{equation}
for $w\in\C\setminus\{-k\mid k\in\N\}$ and $z\in\C$, 
where $\Gamma(\cdot)$ is the Gamma function.
If $\infty$ appears in the denominator of the right-hand side of equation~\Eq{binom}, $\binom{w}{z}$ is regarded as $0$.

When $\a=1$, the equality holds in 
\Eq{lyons}, which is just the conventional binomial theorem.
Therefore, \Thm{neo1} is regarded as a generalisation of the binomial theorem.
The proof of \Thm{neo1} in \cite{Ly98} is rather technical and is derived from the maximum principle for sub-parabolic functions.
Based on the numerical evidence, Lyons conjectured that coefficient $\a^2$ in the left-hand side of 
inequality~\Eq{lyons} could be replaced by $\a$.
Thus far, only partial positive answers were known: E.~R.~Love~\cite{Lo98,Lo00} proved that the conjecture is true for $\a=2^{-k}$ ($k=1,2,3,\dots$)
and that coefficient $\a^2$ in 
\Eq{lyons} can be replaced by $\a/2$ in general.
His proof is based on the duplication formula of the Gamma function; it seems difficult to use his method to provide a complete answer to the conjecture.
Some detailed calculation along the lines of his method has also been carried out in \cite{Li04}.

In this paper, we provide an affirmative answer to Lyons' conjecture with more explicit information. 
Our method is different from the methods mentioned above; in our method, we use the basic theory of complex analysis.
The answer to the conjecture is stated as follows.
\begin{thm}\label{th:main}
Let $\a\in(0,1]$, $n\in\N$, $x\ge0$, and $y\ge0$.
Then, we have
\begin{equation}\label{eq:neo1}
\a\sum_{j=0}^n\binom{\a n}{\a j}x^{\a j}y^{\a(n-j)}\le (x+y)^{\a n}.
\end{equation}
The equality holds if and only if $\a=1$ or $x=y=0$.
\end{thm}

When $\a=1$ or $x=y=0$ holds, it is evident that 
inequality~\Eq{neo1} holds with equality. Moreover, when $0<\a<1$ and only one of $x$ and $y$ is $0$, 
inequality~\Eq{neo1} is trivial with strict inequality.
Therefore, it is sufficient to prove 
\Eq{neo1} with strict inequality for $\a\in(0,1)$, $x>0$, and $y>0$.
We may assume that $x\le y$ by symmetry.
By dividing both sides by $y^{\a n}$ and letting $\lm=x/y$,
it is sufficient to prove the following:
for $\a\in(0,1)$, $n\in\N$, and $0<\lm\le 1$,
\begin{equation}\label{eq:neo2}
\a\sum_{j=0}^n\binom{\a n}{\a j}\lm^{\a j}<(1+\lm)^{\a n}.
\end{equation}
In fact, we can prove the following identity.
\begin{thm}[(Generalisation of the binomial theorem)]
\label{th:main2}
Let $\a\in(0,2)$, $n\in\N$, and $0<\lm\le 1$.
Then, we have
\begin{align}\label{eq:neo3}
\a\sum_{j=0}^n\binom{\a n}{\a j}\lm^{\a j}=(1+\lm)^{\a n}-{}&\frac{\a\lm^\a\sin \a\pi}{\pi}\int_0^1 t^{\a-1}(1-t)^{\a n}\nonumber\\
&\times\left\{\frac{1}{|t^\a-\lm^\a e^{-i\a\pi}|^2}+\frac{\lm^{\a n}}{|e^{-i\a\pi}-(\lm t)^\a|^2}\right\}dt.
\end{align}
\end{thm}
Since the right-hand side of equation~\Eq{neo3} is clearly less than $(1+\lm)^{\a n}$ for $\a\in(0,1)$, \Thm{main2} immediately implies that 
inequality~\Eq{neo2} is valid; therefore, \Thm{main} is proved.
When $\a\in(1,2)$, the right-hand side of 
\Eq{neo3} is greater than $(1+\lm)^{\a n}$.
Therefore, in this case, the converse inequalities of 
\Eq{neo2} and \Eq{neo1} are known as the byproducts.
\Thm{main2} is proved by the application of the fractional order Taylor-like expansions 
with residual terms, which are obtained from the basic theory of complex analysis.

To show that coefficient $\a$ is the best constant, we denote the right-hand side of equation~\Eq{neo3} by $(1+\lm)^{\a n}-R(\a,n,\lm)$.
Then, for fixed $\a\in(0,1)$ and $\lm\in(0,1]$, the error term $R(\a,n,\lm)$ monotonically converges to $0$ as $n\to\infty$ since the integrand decreases to $0$ pointwisely.
In this sense, the constant $\a$ in the left-hand sides of 
\Eq{neo1} and \Eq{neo2} is optimal.
We can also prove that $R(\a,n,\lm)$ is uniformly bounded in $\a\in(0,1]$, $n\in\N$, and $\lm\in(0,1]$ (see \Prop{bound} below).

This paper is 
organised as follows.
In Section~2, we introduce the fractional order Taylor's series and prove \Thm{main2}.
In Section~3, we discuss some 
generalisations of the main theorems.

\section{Fractional order Taylor's series and proof of \Thm{main2}}
\label{sect:2}
Let $D=\{z\in\C\mid |z|<1\}$ be the unit disk in $\C$, and $\bar D$ its closure.
Let $f$ be a continuous function on $\bar D$ such that $f$ is holomorphic in $D$.

For each $\xi\in\R$, we define
\begin{align}
f^\#(\xi)&:=\int_{-1/2}^{1/2}f(e^{2\pi ix})e^{-2\pi ix\xi}\,dx \label{eq:derivative1}\\
&=\frac1{2\pi i}\int_C \frac{f(z)}{z^{\xi+1}}\,dz, \label{eq:derivative2}
\end{align}
where $C$ denotes the oriented contour $(-1,1)\ni t\mapsto\exp(i\pi t)\in\C$.
In \Eq{derivative2}, $z^{\xi+1}$ is defined as $\exp\{(\xi+1)\Log z\}$ on $\C\setminus\{z\in\R\mid z\le0\}$,
where the branch of $\Log$ is taken so that $\Log 1=0$.
It should be noted that $f^\#(\xi)$ is a bounded function in $\xi$.
Then, we have the fractional order Taylor-like series of $f$ with residual terms as follows:
\begin{thm}\label{th:taylor}
For $0<\a<2$ and $0<\lm<1$,
the following identities hold:
\begin{align}\label{eq:taylor1}
\a\sum_{j=0}^\infty f^\#(\a j)\lm^{\a j}
&=f(\lm)
-\frac{\a\lm^\a\sin \a\pi}{\pi}\int_0^1\frac{t^{\a-1}f(-t)}{|t^\a-\lm^\a e^{-i\a\pi}|^2}\,dt,\\
\a\sum_{j=-\infty}^{-1} f^\#(\a j)\lm^{-\a j}
&=\frac{\a\lm^\a\sin \a\pi}{\pi}\int_0^1\frac{t^{\a-1}f(-t)}{|e^{-i\a\pi}-(\lm t)^\a |^2}\,dt.\label{eq:taylor2}
\end{align}
In particular, we have
\begin{align}
&\a\sum_{j=-\infty}^\infty f^\#(\a j)\lm^{\a |j|}\nonumber\\
&=f(\lm)
-\frac{\a\lm^\a(1-\lm^{2\a})\sin \a\pi}{\pi}\int_0^1\frac{t^{\a-1}(1-t^{2\a})f(-t)}{|(t^\a-\lm^\a e^{-i\a\pi})(e^{-i\a\pi}-(\lm t)^\a)|^2}\,dt. \label{eq:taylor3}
\end{align}
If, in addition,
\begin{equation}\label{eq:summable}
\sum_{j=-\infty}^\infty |f^\#(\a j)|<\infty,
\end{equation}
the identities above are valid for $\lm=1$; especially, equation~\Eq{taylor3} becomes
\begin{equation}\label{eq:osler}
\a\sum_{j=-\infty}^\infty f^\#(\a j)=f(1).
\end{equation}
\end{thm}
\begin{rem}
From expression \Eq{derivative2}, for $\xi\in\Z$, we have
\[
  f^\#(\xi)=\begin{cases}
  \displaystyle\left.\frac{d^\xi f}{dz^\xi}(0)\right/\xi! & (\xi\ge0),\\
  0 & (\xi<0).
  \end{cases}
\]
Therefore, when $\a=1$, equation~\Eq{taylor1} is identical to the Taylor series expansion of $f$ at $z=0$, and equation~\Eq{taylor2} is reduced to $0=0$.

When $\xi\notin\Z$, $\Gamma(\xi+1)f^\#(\xi)$ is regarded as a sort of `$\xi$-order' fractional derivative of $f$ at $0$, 
which is denoted by $D^\xi_{z+1} f(0)$ in some literatures. 
(It should be noted that we can transform the contour $C$ in \Eq{derivative2} homotopically in 
$\bar D\setminus \{z\in\R\mid z\le0\}$; however, the terminal points $-1$ should be fixed. This is the reason why $+1\,(=-(-1))$ is specified in the symbol $D^\xi_{z+1}$.)
Fractional order Taylor's series have been considered in various frameworks with a variety of fractional order derivatives. 
(For example, see \cite{Ju06,Os70,Os71,Os73} and the references therein.)
\Thm{taylor} is regarded as another variant.
Equation~\Eq{osler} is consistent with the results obtained by Osler~\cite{Os71}.
It should be noted that equation~\Eq{osler} can also be directly obtained by using Poisson's summation formula under the appropriate integrability condition on $f^\#$.
\end{rem}
\begin{rem}
A simple sufficient condition of \Eq{summable} is 
\begin{equation}\label{eq:smooth}
f\in C^2(S^1)\mbox{ and }f(-1)=0,
\end{equation}
where $S^1$ is the unit circle in $\C$. 
Indeed, by letting $h(x)=f(e^{2\pi ix})$ for $x\in[-1/2,1/2]$, the integration by parts in equation~\Eq{derivative1} implies that
\begin{align*}
f^\#(\xi)
&=\left[h(x)\cdot\frac{e^{-2\pi i\xi x}}{-2\pi i\xi}\right]_{x=-1/2}^{x=1/2}
-\int_{-1/2}^{1/2}\frac{h'(x)}{-2\pi i\xi}\left(\frac{e^{-2\pi i\xi x}}{-2\pi i\xi}\right)'dx\\
&=0+
\frac{1}{4\pi^2\xi^2}\left[h'(x)\cdot e^{-2\pi i\xi x}\right]_{x=-1/2}^{x=1/2}
-\frac{1}{4\pi^2\xi^2}\int_{-1/2}^{1/2}h''(x)\cdot e^{-2\pi i\xi x}\,dx\\
&=O(\xi^{-2}) \qquad(|\xi|\to\infty).
\end{align*}

\end{rem}
\begin{proof}[of \Thm{taylor}]
First, we prove equation~\Eq{taylor1}. Since $f$ is bounded on $C$, we have
\begin{align}
\a\sum_{j=0}^\infty f^\#(\a j)\lm^{\a j}
&=\a\sum_{j=0}^\infty\left( \frac1{2\pi i}\int_C \frac{f(z)}{z^{\a j+1}}\,dz\right)\lm^{\a j}\nonumber\\
&=\frac{\a}{2\pi i}\int_C f(z) \left(\sum_{j=0}^\infty z^{-\a j-1}\lm^{\a j}\right)dz\nonumber\\
&=\frac{\a}{2\pi i}\int_C f(z) \frac{z^{\a-1}}{z^{\a}-\lm^\a}\,dz. \label{eq:integral}
\end{align}
We define
\[
  g(z)=\frac{\a}{2\pi i}\cdot f(z) \cdot\frac{z^{\a-1}}{z^{\a}-\lm^\a}
\]
and consider the contour $\Gamma$ described in Figure~\ref{fig:1}.
More specifically, $C'$, $\Gamma_1$, and $\Gamma_2$ are defined as
\begin{align*}
C'&\colon (-1,1)\ni t\mapsto \eps\exp(-i \pi t)\in\C,\\
\Gamma_1&\colon [-1,-\eps]\ni t\mapsto t+i0\in\C,\\
\Gamma_2&\colon [\eps,1]\ni t\mapsto -t-i0\in\C
\end{align*}
for $\eps\in(0,\lm)$.
\begin{figure}[htbp]
\centering
\includegraphics[scale=.9]{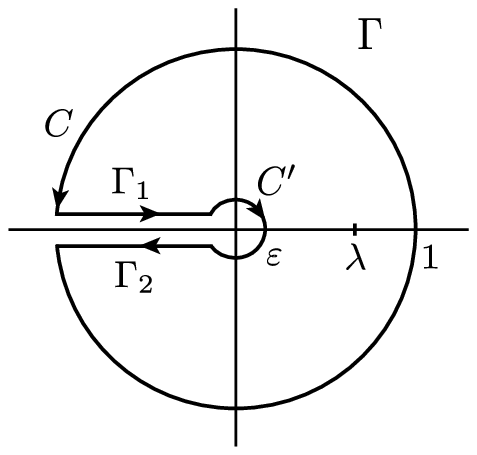}
\caption{Contour $\Gamma=C\cup \Gamma_1\cup C'\cup \Gamma_2$}\label{fig:1}
\end{figure}

In the domain surrounded by the contour $\Gamma$, the function $g$ is holomorphic except at $\lm$ and has the at most first-order pole at $\lambda$. 
(Here, we used the assumption that $0<\a<2$.)
Indeed, we have
\begin{align*}\label{eq:pole}
(z-\lm)g(z)
&=\frac{\a}{2\pi i}\cdot f(z)\cdot\frac{(z-\lm)z^{\a-1}}{z^\a-\lm^\a}\nonumber\\
&\xrightarrow{z\to\lm}
\frac{\a}{2\pi i}\cdot f(\lm)\cdot\frac{\lm^{\a-1}}{(z^\a)'|_{z=\lm}}\nonumber\\
&=\frac{f(\lm)}{2\pi i},
\end{align*}
and the residue of $g$ at $\lm$ is $f(\lm)/(2\pi i)$.
From the residue theorem, we have
\begin{equation}\label{eq:int1}
  \int_{\Gamma}g(z)\,dz
  =2\pi i\cdot\frac{f(\lm)}{2\pi i}=f(\lm).
\end{equation}
On the circle $\{z\in\C\mid |z|=\eps\}$,
\[
|g(z)|
\le \frac{\a}{2\pi}\cdot|f(z)|\cdot\frac{\eps^{\a-1}}{\lm^\a-\eps^\a}=O(\eps^{\a-1})
\quad (\eps\to0).
\]
Therefore,
\begin{equation}\label{eq:int2}
\left|\int_{C'}g(z)\,dz\right|\le\int_{C'}|g(z)|\,|dz|=O(\eps^\a)=o(1)
\quad(\eps\to0).
\end{equation}
Moreover, we have
\begin{align*}
\int_{\Gamma_1}g(z)\,dz
&=\int_1^\eps g(te^{i(\pi-0)})e^{i\pi}\,dt
\qquad(\mbox{by the substitution }z=e^{i(\pi-0)}t)\\
&=\frac{\a}{2\pi i}\int_1^\eps f(-t)\cdot\frac{-t^{\a-1} e^{i\a\pi}}{t^\a e^{i\a\pi}-\lm^\a }\cdot(-1)\,dt\\
&=-\frac{\a}{2\pi i}\int_\eps^1 f(-t)\frac{t^{\a-1}}{t^\a-\lm^\a e^{-i\a\pi}}\,dt
\end{align*}
and
\begin{align*}
\int_{\Gamma_2}g(z)\,dz
&=\int_\eps^1 g(te^{-i(\pi-0)})e^{-i\pi}\,dt
\qquad(\mbox{by the substitution }z=e^{-i(\pi-0)}t)\\
&=\frac{\a}{2\pi i}\int_\eps^1 f(-t)\cdot\frac{-t^{\a-1} e^{-i\a\pi}}{t^\a e^{-i\a\pi}-\lm^\a }\cdot(-1)\,dt\\
&=\frac{\a}{2\pi i}\int_\eps^1 f(-t)\frac{t^{\a-1}}{t^\a-\lm^\a e^{i\a\pi}}\,dt.\end{align*}
Therefore,
\begin{align}\label{eq:int3}
\int_{\Gamma_1\cup\Gamma_2}g(z)\,dz
&=\frac{\a}{2\pi}\int_\eps^1 f(-t)\cdot t^{\a-1}\cdot 2\Re\left[\frac{-1}i\cdot\frac{1}{t^\a-\lm^\a e^{-i\a\pi}}\right]\,dt\nonumber\\
&=\frac{\a}{\pi}\int_\eps^1 f(-t)\cdot t^{\a-1}\cdot \Im\left[\frac{-1}{t^\a-\lm^\a e^{-i\a\pi}}\right]\,dt\nonumber\\
&=\frac{\a}{\pi}\int_\eps^1 f(-t)\cdot t^{\a-1}\cdot \frac{\lm ^\a \sin \a \pi}{|t^\a-\lm^\a e^{-i\a\pi}|^2}\,dt.
\end{align}
Combining equations~\Eq{integral}--\Eq{int3} and letting $\eps\to 0$, we obtain equation~\Eq{taylor1}.

Equation~\Eq{taylor2} is proved along the same lines as the proof of equation~\Eq{taylor1}.
However, in this case, we use the following equality instead of \Eq{integral}:
\begin{align*}
\a\sum_{j=-\infty}^{-1} f^\#(\a j)\lm^{-\a j}
&=\a\sum_{j=-\infty}^{-1}\left(\frac{1}{2\pi i}\int_C \frac{f(z)}{z^{\a j+1}}\,dz\right)\lm^{-\a j}\nonumber\\
&=\frac{\a}{2\pi i}\int_C f(z) \left(\sum_{k=1}^\infty z^{\a k-1}\lm^{\a k}\right)dz\nonumber\\
&=\frac{\a}{2\pi i}\int_C f(z) \frac{z^{\a-1}\lm^\a}{1-z^{\a}\lm^\a}\,dz. 
\end{align*}
The integrand is holomorphic in the domain surrounded by $\Gamma$.
Therefore, we have
\[
\frac{\a}{2\pi i}\int_{\Gamma}f(z)\frac{z^{\a-1}\lm^\a}{1-z^\a\lm^\a}\,dz=0.
\]
Instead of equation~\Eq{int3}, we consider
\begin{align*}
\frac{\a}{2\pi i}\int_{\Gamma_1\cup\Gamma_2}f(z)\frac{z^{\a-1}\lm^\a}{1-z^\a\lm^\a}\,dz
&=\frac{\a}{2\pi}\int_\eps^1 f(-t)\cdot t^{\a-1}\lm^\a\cdot 2\Im\left[\frac{-1}{e^{-i\a\pi}-t^\a\lm^\a}\right]dt\\
&=\frac{\a}{\pi}\int_\eps^1 f(-t)\cdot t^{\a-1}\lm^\a\cdot \frac{-\sin\a\pi}{|e^{-i\a\pi}-(\lm t)^\a|^2}\,dt.
\end{align*}
The other calculations are carried out in the same manner as the proof of equation~\Eq{taylor1}.

Equation~\Eq{taylor3} is simply the sum of equations~\Eq{taylor1} and \Eq{taylor2}, because\begin{align*}
&-\frac{1}{|t^\a-\lm^\a e^{-i\a\pi}|^2}+\frac1{|e^{-i\a\pi}-(\lm t)^\a|^2}\\
&=\frac{-(1-2(\lm t)^\a\cos \a\pi+(\lm t)^{2\a})+(t^{2\a}-2t^\a \lm^\a\cos\a\pi+\lm^{2\a})}{|t^\a-\lm^\a e^{-i\a\pi}|^2|e^{-i\a\pi}-(\lm t)^\a|^2}\\
&=-\frac{(1-\lm^{2\a})(1-t^{2\a})}{|(t^\a-\lm^\a e^{-i\a\pi})(e^{-i\a\pi}-(\lm t)^\a)|^2}.
\end{align*}

When equation~\Eq{summable} holds, by taking the limit $\lm\uparrow1$ in equations~\Eq{taylor1}--\Eq{taylor3}, the dominated convergence theorem assures that these equations are also true for $\lm=1$.
\end{proof}
The following is another key fact for the proof of \Thm{main2}.
\begin{prop}\label{prop:nikou}
Let $T>0$ and define $f(z)=(1+z)^T$ on $\bar D$.
Then, the function $f^\#$ that is defined in 
\Eq{derivative1} is expressed as
\begin{equation}\label{eq:nikou}
f^\#(\xi)=\binom{T}{\xi}\quad \mbox{for }\xi\in\R.
\end{equation}
In particular, we have
\begin{equation}\label{eq:symmetry}
f^\#(\xi)=f^\#(T-\xi)\quad \mbox{for }\xi\in\R.
\end{equation}
\end{prop}
This is a classical result; for example, see \cite{Po72} and the references therein for the proof.
It should be noted that equation~\Eq{symmetry} is evident from expression~\Eq{nikou}; however, it is also directly proved by the definition of $f$ and $f^\#$ and the change of variables.

Now, we prove \Thm{main2}.
\begin{proof}[of \Thm{main2}]
Let $0<\a<2$ and $n\in\N$.
We define $f(z)=(1+z)^{\a n}$ on $\bar D$.
First, assume that $0<\lm<1$.
From \Prop{nikou}, we have
\begin{align}\label{eq:henkei}
\a\sum_{j=0}^n \binom{\a n}{\a j}\lm^{\a j}
&=\a\sum_{j=0}^n f^\#(\a j)\lm^{\a j}\nonumber\\
&=\a\sum_{j=0}^\infty f^\#(\a j)\lm^{\a j}-\a\sum_{j=n+1}^\infty f^\#(\a j)\lm^{\a j}\nonumber\\
&=\a\sum_{j=0}^\infty f^\#(\a j)\lm^{\a j}-\a\sum_{k=-\infty}^{-1} f^\#(\a k)\lm^{-\a k}\lm^{\a n}.
\end{align}
In the last equality, we substituted $k$ for $n-j$ and used the relation $f^\#(\a n-\a k)=f^\#(\a k)$ that is derived from equation~\Eq{symmetry}.
Applying equations~\Eq{taylor1} and \Eq{taylor2} in \Thm{taylor} to \Eq{henkei}, we obtain the identity \Eq{neo3} in \Thm{main2} for $\lm\in(0,1)$.
From the dominated convergence theorem, we can take the limit $\lm\uparrow1$ to conclude that this equation is still true for $\lm=1$.
\end{proof}
\begin{rem}
By using the functional equality $\Gamma(z)\Gamma(1-z)={\pi}/(\sin \pi z)$ and the Stirling formula 
$\Gamma(x)/(\sqrt{2\pi}e^{-x}x^{x-(1/2)})\to1$ 
$(x\in\R,\ x\to+\infty)$, we can prove that 
\[
\binom{T}{\xi}=O(|\xi|^{-T-1})
\quad (\xi\in\R,\ |\xi|\to\infty)
\]
for $T>0$.
Therefore, the condition \Eq{summable} holds for $f(z)=(1+z)^{\a n}$. 
If we use this fact, we do not need to make an exception the case $\lm=1$ in the proof of \Thm{main2}.
Moreover, equation~\Eq{osler} implies that for $0<\a<2$,
\[
\a\sum_{j=-\infty}^\infty \binom{\a n}{\a j}=2^{\a n}.
\]
This identity is a special case of more general results obtained by Osler~\cite[Eq.~(5.1)]{Os71}.
\end{rem}
At the end of this section, we provide a quantitative estimate of the error term $R(\a,n,\lm)$ mentioned in Section~1; that is,
\begin{equation}\label{eq:R}
R(\a,n,\lm):=\frac{\a\lm^\a\sin \a\pi}{\pi}\int_0^1 t^{\a-1}(1-t)^{\a n}\left\{\frac{1}{|t^\a-\lm^\a e^{-i\a\pi}|^2}+\frac{\lm^{\a n}}{|e^{-i\a\pi}-(\lm t)^\a|^2}\right\}dt.
\end{equation}
Since $R(1,n,\lm)\equiv0$, we can suppose that $\a\ne1$.
\begin{prop}\label{prop:bound}
Let $n\in\N$ and $\lm\in(0,1]$.
Then, we have 
\begin{align*}
0&< R(\a,n,\lm)<1-\a<1\phantom{-1}
\mbox{ for }\a\in(0,1)
\intertext{and}
0&> R(\a,n,\lm)>1-\a>-1\phantom{1}
\mbox{ for }\a\in(1,2).
\end{align*}
\end{prop}
\begin{proof}
First, we suppose that $0<\a<1$.
Then, we have
\begin{align*}
0&< R(\a,n,\lm) \le R(\a,1,\lm)
\qquad\mbox{(from \Eq{R})}\\
&=(1+\lm)^\a-\a\sum_{j=0}^1\binom{\a}{\a j}\lm^{\a j}
\qquad\mbox{(from \Thm{main2})}\\
&=(1+\lm)^\a-\a(1+\lm^{\a})\\
&=\{(1+\lm)^\a-(1+\a\lm)\}+\a(\lm-\lm^\a)+(1-\a).
\end{align*}
Since $(1+\lm)^\a$ is strictly concave in $\lm$ and its derivative at $\lm=0$ is $\a$, the first term in the above equation is less than $0$.
The second term is also dominated by $0$ for $\lm\in(0,1]$.
Therefore, the above equation is less than $1-\a\,(<1)$.
 
The case that $1<\a<2$ is proved in the same manner; in this case, we have
\begin{align*}
0&> R(\a,n,\lm) \ge R(\a,1,\lm)\\
&=\{(1+\lm)^\a-(1+\a\lm)\}+\a(\lm-\lm^\a)+(1-\a).
\end{align*}
Since $(1+\lm)^\a$ is strictly convex in $\lm$, the first term in the above equation is greater than $0$.
The second term is greater than or equal to $0$.
Therefore, the above equation is greater than $1-\a\,(>-1)$.
\end{proof}

\section{Some generalisations}\label{sect:3}
In this section, we discuss some generalisations of Theorems~\ref{th:taylor} and \ref{th:main2}.
For $z\in\C\setminus\{0\}$, we express $z$ as $z=r e^{i\theta}$ ($r>0$, $\theta\in(-\pi,\pi]$) and define $z^\b$ as $z^\b=r^{\b}e^{i\theta \b}$ for $\b\in\R$ as a convention.
The point is that the argument of $z$ is taken in the interval $(-\pi,\pi]$.

For $\a>0$, we define
\[
  K_\a:=\{\om\in\C\mid \om^\a=1\}
  \,(=\{e^{i\theta}\mid -\pi<\theta\le\pi\mbox{ and } e^{i\theta\a}=1\}).
\]
It should be noted that $K_\a=\{1\}$ when $0<\a<2$ and that $-1\in K_\a$ if and only if $\a/2\in\N$.
Then, Theorems~\ref{th:taylor} and \ref{th:main2} are generalised as follows.
\begin{thm}\label{th:taylorA}
Let $f$ be a continuous function on $\bar D$ that is holomorphic in $D$.
Let $\a>0$, $\gm<\a$, and $0<\lm<1$. 
Suppose that $\a/2\notin\N$.
Then, the following identities hold:
\begin{align}\label{eq:taylor1A}
\a\sum_{j=0}^\infty f^\#(\a j+\gm)\lm^{\a j}
&=\sum_{\om\in K_\a}(\lm \om)^{-\gm}f(\lm \om)\nonumber\\
&\quad-\frac{\a}{\pi}\int_0^1 f(-t)t^{\a-\gm-1}\frac{t^\a\sin\gm\pi+\lm^\a\sin(\a-\gm)\pi}{|t^\a-\lm^\a e^{-i\a\pi}|^2}\,dt,\\
\a\sum_{j=-\infty}^{-1} f^\#(\a j+\gm)\lm^{-\a j}
&=\frac{\a\lm^{\a}}{\pi}\int_0^1 f(-t)t^{\a-\gm-1}\frac{\sin(\a-\gm)\pi+(\lm t)^\a\sin\gm\pi}{|e^{-i\a\pi}-(\lm t)^\a |^2}\,dt.\label{eq:taylor2A}
\end{align}
In particular, we have
\begin{align}
&\a\sum_{j=-\infty}^\infty f^\#(\a j+\gm)\lm^{\a |j|}\nonumber\\
&=\sum_{\om\in K_\a}(\lm \om)^{-\gm}f(\lm \om)
-\frac{\a(1-\lm^{2\a})}{\pi}\int_0^1 f(-t)t^{\a-\gm-1}\nonumber\\
&\qquad\times\frac{\lm^\a(1-t^{2\a})\sin(\a-\gm)\pi+t^\a(1-2(\lm t)^\a\cos \a\pi+\lm^{2\a})\sin\gm\pi}{|(t^\a-\lm^\a e^{-i\a\pi})(e^{-i\a\pi}-(\lm t)^\a )|^2}\,dt. \label{eq:taylor3A}
\end{align}
If, in addition,
\begin{equation}\label{eq:summableA}
\sum_{j=-\infty}^\infty |f^\#(\a j+\gm)|<\infty,
\end{equation}
then, the identities \Eq{taylor1A}--\Eq{taylor3A} are valid for $\lm=1$; in particular, equation~\Eq{taylor3A} becomes
\begin{equation}\label{eq:oslerA}
\a\sum_{j=-\infty}^\infty f^\#(\a j+\gm)=\sum_{\om\in K_\a}\om^{-\gm}f(\om).
\end{equation}
When $\a/2\in\N$ and $\gm=0$, the above facts still hold by regarding the second terms of the right-hand sides of equations~\Eq{taylor1A} and \Eq{taylor3A} and the right-hand side of equation~\Eq{taylor2A} as zero.
\end{thm}
\begin{thm}\label{th:main2A}
Let $\a>0$, $n\in\N$, and $0<\lm\le 1$.
Then, we have
\begin{align}\label{eq:neo3A}
\a\sum_{j=0}^n\binom{\a n}{\a j}\lm^{\a j}&=\sum_{\om\in K_\a}(1+\lm \om)^{\a n}-\frac{\a\lm^{\a}\sin \a\pi}{\pi}\int_0^1 t^{\a-1}(1-t)^{\a n}\nonumber\\
&\qquad\times\left\{\frac{1}{|t^\a-\lm^\a e^{-i\a\pi}|^2}+\frac{\lm^{\a n}}{|e^{-i\a\pi}-(\lm t)^\a |^2}\right\}\,dt.
\end{align}
Here, the second term on the right-hand side of this equation is regarded as zero if $\a/2\in\N$.

In particular, we have
\begin{align*}
\a\sum_{j=0}^n\binom{\a n}{\a j}\lm^{\a j}-\sum_{\om\in K_\a}(1+\lm \om)^{\a n}\begin{cases}
<0& \mbox{if }2m<\a<2m+1\mbox{ for some }m\in\N\cup\{0\},\\
=0& \mbox{if }\a\in\N,\\
>0& \mbox{if }2m+1<\a<2m+2\mbox{ for some }m\in\N\cup\{0\}.
\end{cases}
\end{align*}
\end{thm}
\begin{rem}
\begin{enumerate}
\item Note that $\sum_{\om\in K_\a}(1+\lm \om)^{\a n}$ in equation~\Eq{neo3A} is a real number.
\item As in the case of \Thm{taylorA}, we can introduce 
the parameter
$\gm$ in \Thm{main2A}. 
However, since the introduction of $\gm$ only makes the equations complicated, we did not include it.
\end{enumerate}
\end{rem}
\begin{proof}[of \Thm{taylorA}]
The concept of this proof is the same as that of \Thm{taylor}.
We adopt the symbols used there.
Suppose that $0<\lm<1$.
Then, we have
\begin{align}\label{eq:sum1}
\a\sum_{j=0}^\infty f^\#(\a j+\gm)\lm^{\a j}
&=\a\sum_{j=0}^\infty\left(\frac1{2\pi i}\int_C \frac{f(z)}{z^{\a j+\gm+1}}\,dz\right)\lm^{\a j}\nonumber\\
&=\frac{\a}{2\pi i}\int_C f(z)\left(\sum_{j=0}^\infty z^{-\a j-\gm-1} \lm^{\a j}\right)\,dz\nonumber\\
&=\frac{\a}{2\pi i}\int_C f(z)\frac{z^{\a-\gm-1}}{z^{\a}-\lm^\a}\,dz.
\end{align}
Define 
\[
g(z)=\frac{\a}{2\pi i}\cdot f(z)\cdot\frac{z^{\a-\gm-1}}{z^{\a}-\lm^\a}.
\]
Suppose that $\a/2\notin\N$.
Then, $g$ is meromorphic in the domain surrounded by $\Gamma$, as shown in Figure~\ref{fig:1}.
All the poles of $g$ are included in $\{\lm \om\mid \om\in K_\a\}$.
In particular, the poles do not exist on $\Gamma_1\cup \Gamma_2$.
Since for $\om\in K_\a$, we have
\begin{align*}
(z-\lm \om)g(z)\xrightarrow{z\to\lm \om}
&\frac{\a}{2\pi i}\cdot f(\lm \om)\cdot\frac{(\lm \om)^{\a-\gm-1}}{(z^\a)'|_{z=\lm \om}}\\
=&
\;\frac{(\lm \om)^{-\gm}f(\lm \om) }{2\pi i},
\end{align*}
the residue of $g$ at $\lm \om$ is $(\lm \om)^{-\gm}f(\lm \om)/(2\pi i)$.
From the residue theorem, we have
\[
\int_{\Gamma}g(z)\,dz=\sum_{\om\in K_\a}(\lm \om)^{-\gm}f(\lm \om).
\]
As in the proof of \Thm{taylor}, we can prove that
\[
  \int_{C'}g(z)\,dz=O(\eps^{\a-\gm})=o(1)
  \quad\mbox{as }\eps\to0
\]
and
\begin{align*}
\int_{\Gamma_1\cup\Gamma_2}g(z)\,dz
&=\frac{\a}{\pi}\int_\eps^1 f(-t)\cdot t^{\a-\gm-1}
\Im\left[\frac{-e^{-i\gm\pi}}{t^\a-\lm^\a e^{-i\a\pi}}\right]dt\\
&=\frac{\a}{\pi}\int_\eps^1 f(-t)\cdot t^{\a-\gm-1}\cdot
\frac{t^\a\sin\gm\pi+\lm^\a\sin(\a-\gm)\pi}{|t^\a-\lm^\a e^{-i\a\pi}|^2}\,dt.
\end{align*}
These equations imply that equation~\Eq{taylor1A} is valid.
Similarly, we have
\begin{align}\label{eq:sum2}
\a\sum_{j=-\infty}^{-1} f^\#(\a j+\gm)\lm^{-\a j}
&=\a\sum_{j=-\infty}^{-1}\left(\frac{1}{2\pi i}\int_C \frac{f(z)}{z^{\a j+\gm-1}}\,dz\right)\lm^{-\a j}\nonumber\\
&=\frac{\a}{2\pi i}\int_C f(z) \left(\sum_{k=1}^\infty z^{\a k-\gm-1}\lm^{\a k}\right)dz\nonumber\\
&=\frac{\a}{2\pi i}\int_C f(z) \frac{z^{\a-\gm-1}\lm^\a}{1-z^{\a}\lm^\a}\,dz. 
\end{align}
Since the integrand is holomorphic in the domain surrounded by $\Gamma$, we have
\[
\frac{\a}{2\pi i}\int_{\Gamma}f(z)\frac{z^{\a-\gm-1}\lm^\a}{1-z^\a\lm^\a}\,dz=0.
\]
It also follows that the integral along $C'$ is negligible as $\eps\to0$, and 
\begin{align*}
\frac{\a}{2\pi i}\int_{\Gamma_1\cup\Gamma_2}f(z)\frac{z^{\a-\gm-1}\lm^\a}{1-z^\a\lm^\a}\,dz
&=\frac{\a}{\pi}\int_\eps^1 f(-t)\cdot t^{\a-\gm-1}\lm^\a\cdot \Im\left[\frac{-e^{-i\gm\pi}}{e^{-i\a\pi}-t^\a\lm^\a}\right]dt\\
&=-\frac{\a}{\pi}\int_\eps^1 f(-t)\cdot t^{\a-\gm-1}\lm^\a\cdot \frac{\sin(\a-\gm)\pi+(\lm t)^\a\sin\gm\pi}{|e^{-i\a\pi}-(\lm t)^\a|^2}\,dt.
\end{align*}
From these calculations, it is inferred that equation~\Eq{taylor2A} holds.
Equation~\Eq{taylor3A} is obtained by adding equations~\Eq{taylor1A} and \Eq{taylor2A}.

Under the condition \Eq{summableA}, we obtain equations~\Eq{taylor1A}--\Eq{taylor3A} with $\lm=1$ by taking the limit $\lm\uparrow1$ and using the dominated convergence theorem.

When $\a/2\in\N$ and $\gm=0$, the integrands on the right-hand sides of equations~\Eq{sum1} and \Eq{sum2} are meromorphic in $D$ and the poles belong to $\{\lm\om\mid \om\in K_\a\}$. 
Therefore, we can directly apply the residue theorem to the integrals along $C$ in equations~\Eq{sum1} and \Eq{sum2} to obtain equations~\Eq{taylor1A}--\Eq{taylor3A}.
\end{proof}
\begin{rem}
The proof also shows that equation~\Eq{taylor2A} holds for $\gm<\a$ and $0<\lm<1$, even when $\a/2\in \N$ and $\gm\ne0$.
\end{rem}
\begin{proof}[of \Thm{main2A}]
This theorem is proved in the same way as \Thm{main2}; in this case, we use \Thm{taylorA} with $\gm=0$ instead of \Thm{taylor}.
\end{proof}
A table of the correspondence between some concrete functions $f$ and $f^\#$ is found in, for example, \cite{Os70,Os71} with a slightly different terminology.
On the basis of this correspondence, \Thm{taylorA} provides a series of nontrivial functional identities.

\affiliationone{%
   Keisuke Hara\\
   ACCESS Co., Ltd.\\
   Hirata Bldg.\\
   2-8-16, Sarugaku-cho, Chiyoda-ku\\
   Tokyo 101-0064\\
   Japan
   \email{Keisuke.Hara@access-company.com}}
\affiliationtwo{%
   Masanori Hino\\
   Graduate School of Informatics\\
   Kyoto University\\
   Kyoto 606-8501\\
   Japan
   \email{hino@i.kyoto-u.ac.jp}}
\end{document}